\documentclass[[11pt,a4paper,reqno]{article}

\usepackage{bm}
\usepackage{graphicx}%
\usepackage{multirow}%
\usepackage{amsmath,amssymb,amsfonts}%
\usepackage{amsthm}%
\usepackage{mathrsfs}%
\usepackage[title]{appendix}%
\usepackage{xcolor}%
\usepackage{textcomp}%
\usepackage{manyfoot}%
\usepackage{booktabs}%
\usepackage{listings}%
\usepackage{geometry}

\def\beq{\begin{equation}}
\def\eeq{\end{equation}}

\newcommand{\R}{{\mathbb R}}

\newcommand{\N}{{\mathbb N}}
\newcommand{\C}{{\mathbb C}}

\newtheorem{theorem}{Theorem}[section]
\newtheorem{definition}{Definition}

\newtheorem{lemma}{Lemma}

\newtheorem{remark}{Remark}[section]

\numberwithin{equation}{section}

\begin{document}

\title{Trap behaviors for Brownian motions}

\author{Raffaela Capitanelli\footnote{
raffaela.capitanelli@uniroma1.it} \, \& Mirko D'Ovidio\footnote{mirko.dovidio@uniroma1.it}\\
 \textit{Sapienza University of Rome}}
\maketitle

\begin{abstract}
    \quad This paper investigates the relationship between the geometric properties of a domain and the diffusion dynamics of Brownian motion, with a specific focus on the phenomenon of \lq\lq trapping"  in terms of  the behavior of stochastic processes.

\end{abstract}

\bigskip


\bigskip



\section{Introduction}

In this paper, we investigate the interplay between the geometric properties of a  domain $D \subset \mathbb{R}^d$ and the properties  of the heat diffusion on it.

 Specifically, we consider the heat equation $\partial_t u = \Delta u$ subject to classical boundary conditions:$$\begin{aligned}
    (\mathbf{P_D}) \quad & \text{Dirichlet conditions:} \quad u|_{\partial D} = 0, \\
    (\mathbf{P_N}) \quad & \text{Neumann conditions:} \quad \partial_{\mathbf{n}} u|_{\partial D} = 0.
\end{aligned}$$

From a probabilistic point of view, the solution $u(t, x)$ with initial datum $u_0 = 1$ is intrinsically linked to the distribution of the first exit time $\zeta$ of the underlying stochastic process.

Under the Dirichlet setting, the process corresponds to a Brownian motion killed upon hitting the boundary, where $\zeta = \tau_{D}^{\text{exit}}$. Our analysis focuses on the heat content $Q(t) = \int_{D} \mathbf{P}_x(\tau_{D}^{\text{exit}} > t) \, dx$. 
For small $t$, the heat loss $|D| - Q(t)$ is known to scale according to the geometry of the boundary. 
We refine the classical expansion   $|D| - Q(t)$ for $t \to 0^+$ interpreting this quantity as a measure of the \lq\lq trapping capacity" of $\partial D.$

Conversely, under Neumann boundary conditions, the process is a reflected Brownian motion ($\zeta = \infty$).
 In this framework, the boundary's influence is encoded in the hitting and exit times of interior subsets.
 We study the expected survival time $\mathbf{E}_x[\tau_{B}^{\text{exit}}]$ for a ball $B \subset \overline{D}$ and the hitting time $\mathbf{E}_x[\tau_{B}^{\text{hit}}]$ for $B \subset D$.
    These functionals provide a sharp characterization of the domain's geometry, revealing how the boundary $\partial D$ acts as a reflecting \lq\lq trap" that confines the process trajectories within specific regions of the state space.

Summarizing,
given the heat equation on $D$,
\begin{align}
    \partial_t u = \Delta u, \quad u_0=1, \quad u(t,x) = \mathbf{P}_x(\zeta > t)
\end{align}
\begin{itemize}
    \item Dirichlet bc on $\partial D$ implies $\zeta=\tau^\text{exit}_D$ and we study $Q(t) = \int_D \mathbf{P}_x(\zeta > t) dx$ for small times. In particular  the quantity $|D| - Q(t)$ for  $t \ll 1$ gives the heat loss on the boundary, that is \lq\lq how much the boundary traps the heat particles"
    \item  Neumann bc on $\partial D$ implies $\zeta = \infty$ and for a ball $B \subset \overline{D}$,  $\mathbf{E}_x[\tau^\text{exit}_B] \leq \infty$ gives information on the $\partial D$ in terms of trap behaviour on $B \cap \overline{D}$
    \item  Neumann bc on $\partial D$ implies $\zeta = \infty$ and for a ball $B \subset D$,  $\mathbf{E}_x[\tau^\text{hit}_B] \leq \infty$ gives information on the $\partial D$ in terms of trap boundary $\partial D.$
\end{itemize}

The study is therefore conducted through two complementary probabilistic lenses: the Dirichlet condition, which models Brownian motion \lq\lq killed" upon hitting the boundary, and the Neumann condition, which describes reflected Brownian motion. We characterize \lq\lq trap domains" as those where the expected hitting time to an internal subset is infinite, providing a rigorous mathematical framework to identify such geometries. Furthermore, we explore the role of \lq\lq sticky" Brownian motions and time-changed diffusions, where the process exhibits delays or prolonged stays in specific regions of the state space or along the boundary.

These theoretical results are applied to irregular  structures, including Koch snowflakes  and Sierpinski-type fractals.
 By establishing integral criteria and  by using non-local operators, we demonstrate how specific geometric parameters determine whether a domain acts as a trap, confining the trajectories and slowing down the global diffusion.

The plan of the paper is the following.

In Section 2, we consider trap domains according the definition in \cite{BCM} and provide some examples. In particular  we obtain a characterisation for  modified Koch snowflakes in Theorem \ref{KM}.

In Section 3 we recall  the definition and some properties of time changed diffusions with a particular insight 
on non-local dynamic boundary value problems. In particular we 
discuss the sticky behaviour in dimension $d=1$, then  for metric graphs and for bounded domains of $\mathbb{R}^d$ with $d>1$.

In Section 4, we  consider trap behaviours and we give a characterization in terms of reflected BM  in  Definition \ref{defTrapb}  and  in terms of killed BM in  Definition \ref{defTraphc}.

In Section 5, we  consider sub-diffusive behaviours and  in Section 6 we provide some examples and applications.

\bigskip

In the following, we use these notations:
\begin{itemize}
    \item $Z$ with infinitesimal generator $(A, D(A))$ is a continuous and symmetric Markov process on a bounded domain $D$;
     \item $X$ is a Markov process with local time $\gamma;$
    \item  $X^+$ is a reflected  Brownian motion  with local time $\gamma^+;$
    \item  $\bar{X}$ is a sticky Brownian motion  with local time $\bar{\gamma};$
    \item  $B=B(x, \varrho)$ is a ball centred at $x$ with radius $\varrho>0;$
    \item  $T_B(Z) = \inf\{t\,:\, Z_t \in \partial B\}$ for the continuous process $Z=\{Z_t\}_{t \geq 0}$ sometimes provided with additional specifications on $Z_0.$
\end{itemize}

\section{Trap domains}
\label{sec:TrapDom}

\subsection{Characterization}

We start with the following definition given in \cite[formulas (1.1) and (1.2)]{BCM} for the Brownian motion (sometimes, we write BM). Let $B \subset D$ be a closed ball with non-zero radius and denote by $T_B (X^+) = \inf\{t \geq 0 : X^+_t \in  \partial B\}$ the first hitting time of $B$ by $X^+$. The following characterization of $D$ does not depend on the choice of $B$. 
\begin{definition}
\label{defTrapd}
We say that $D$ is a \emph{trap domain} for $X^+$ if 
\begin{align*}
\sup_{x \in D \setminus B} \mathbf{E}_x[T_B(X^+)] = \infty.
\end{align*}
Otherwise, we say that $D$ is a \emph{non-trap domain} for $X^+$.
\end{definition}
In Definition \ref{defTrapd}, the random time $T_B=T_B(X^+)$ plays the role of lifetime for the Brownian motion on $\bar{D} \setminus \bar{B}$ reflected on $\partial D \setminus \partial B$ and killed on $\partial B$. On the other hand (as stated in \cite[Lemma 3.2]{BCM}), 
\begin{align}
\mathbf{E}_x[T_B]< \infty \quad \forall\, x \in D \setminus B.
\end{align}
This aspect suggests to pay special attention to the boundary. A process can be trapped on the boundary, it may have infinite lifetime depending on the regularity of $\partial D$. As described in \cite{BCM}, Definition \ref{defTrapd} makes sense only for non-smooth domains. We also observe that for the Brownian motion,  in the $2$-dimensional case for example, every point is almost surely not visited but there exists a random (big and uncountable) set of points visited infinitely often. Thus, the process may spend a very large time starting near the boundary, it therefore appears trapped.

\subsection{Examples of Trap domains}
\label{sec:VariantOmega}

({\it Koch domain}) A discussion on some modified Koch domains which become trap for the BM has been given in  \cite{BCM}.

 The Koch domain with any scale factor is non trap for the BM and we have proved in Proposition 5.2 in \cite{CDJEE} that such a property still holds for the time-changed BM on random Koch domains (whose boundaries have Hausdorff dimensions between 1 and 2). The problem is to determine under which alterations a domain is still non-trap for a given process. Here we define the variant $\Omega_\mathsf{M}$ of the  Koch snowflake $\Omega_\mathsf{K}$ (see \cite{BCM}).

We recall the definition of the Koch curve with endpoints $A = (0, 0),$ and $B = (1, 0)$. We consider the family $\Psi^\alpha=\{\psi^\alpha_1,\dots, \psi^\alpha_4\}$ of contractive similitudes  $\psi^\alpha_i:\C \rightarrow\C$, $i=1,\ldots,4$, with contraction factor $\alpha^{-1},$  $2<\alpha<4$, 
$$\aligned&\psi^\alpha_1(z)=\frac{z}{\alpha},\,\quad\quad\quad\quad\quad\quad\quad\quad\quad\quad\quad\quad
\psi^\alpha_2(z)=\frac{z}{\alpha}e^{i\theta(\alpha)}+\frac{1}{\alpha},\\
&\psi^\alpha_3(z)=\frac{z}{\alpha}e^{-
i\theta(\alpha)}+\frac12+i\sqrt{\frac{1}{\alpha}-\frac{1}{4}},\quad\quad\quad
\psi^\alpha_4(z)=\frac{z-1}{\alpha}+1,\endaligned$$
where $\theta(\alpha)=\arcsin\left( 
\frac{\sqrt{\alpha(4-\alpha)}}{2}\right).$

By the general theory of self-similar fractals (see \cite{Fal}, \cite{Kig}), there exists a unique closed bounded set $K_\alpha$ which is {\it invariant\/} with respect to $\Psi^\alpha$, that is, 
\begin{equation}
K_\alpha=\cup^4_{i=1}\psi^\alpha_i(K_\alpha).
\end{equation}
We recall that $K_\alpha$ supports a unique self-similar Borel measure 
\begin{align}
\label{mu-alpha-def}
\mu_\alpha \; \textrm{which is equivalent to the $d_f-$dimensional Hausdorff measure}
\end{align}
where $d_f=\frac{\log 4}{\log \alpha}$. Let $K^0$ be the line segment of unit length that has as endpoints $A=(0,0)$ and  $B=(1,0)$. We set, for each $n$ in $\N$, 
\begin{equation}
\label{Kn}
 K_\alpha^1=\bigcup_{i=1}^4 \psi^\alpha_i(K^0),\qquad
 K_\alpha^2=\bigcup_{i=1}^4 \psi^\alpha_i(K_\alpha^1), \qquad
 \dots, \qquad
 K_\alpha^{n+1}=\bigcup_{i=1}^4 \psi^\alpha_i(K_\alpha^n);
\end{equation}
${K_\alpha^n}$ is the so-called ${n}${-th} {pre-fractal curve}. Moreover, the iterates $K_\alpha^{n}$ converge to the self-similar set $K_\alpha$ in the Hausdorff metric, when $n$ tends to infinity.

We  now recall  the construction of  the snowflake type fractals.

 \begin{figure}
     \centering
     \includegraphics[width=0.25\linewidth]{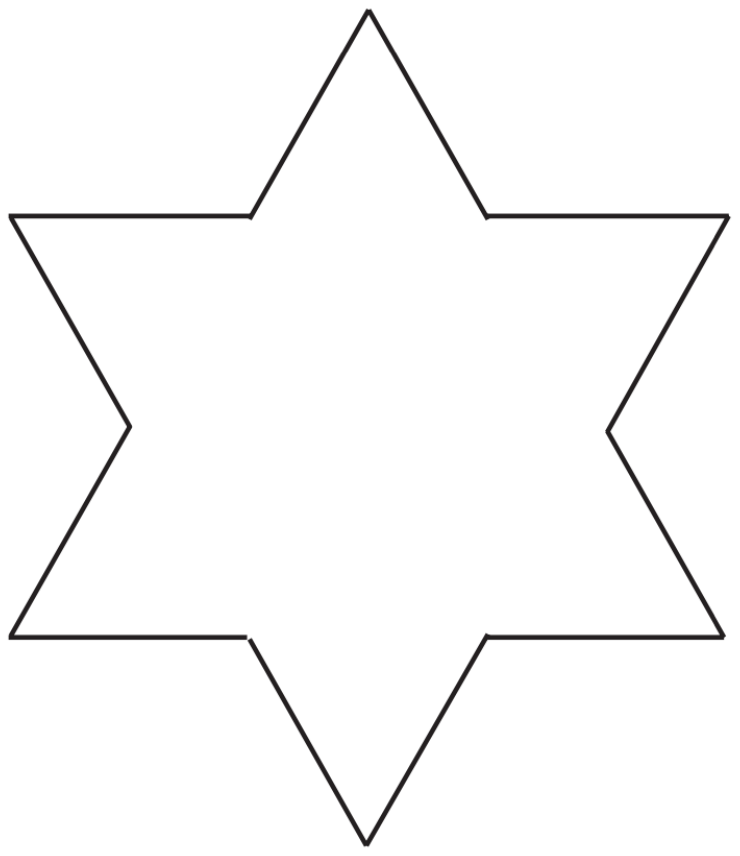}
     \caption{Classical Koch snowflake}
          \label{SNOW}
 \end{figure}

 Let $T_1$ be the triangle  with vertices $A = (0, 0), B = (1, 0),$ and $C = (\frac12, -\frac{\sqrt{3}}{2})$. We construct on the side with endpoints $A$ and $B$ the pre-fractal Koch curve defined before, which will be denoted by  $K^n_{1,\alpha}$ and the Koch curve defined before, which will be denoted by $K_{1,\alpha}$. In a similar way, we construct on the  other sides  the analogous pre-fractal Koch curves (the Koch curves) denoting  by   $K^n_{2,\alpha}$  and $K^n_{3,\alpha}$ (by $K_{2,\alpha}$ and $K_{3,\alpha}$) the  curves with endpoints $B$ and $C$, and $C$ and $A$, respectively.  We denote by  $\Omega_\alpha^n$  the pre-fractal domain that is the set bounded by the pre-fractal Koch curves $K^n_{ j,_\alpha},$  $j=1,2,3.$ Moreover,  we denote by $\Omega_\alpha$  the snowflake that is the set  bounded by the Koch curves $K_{ j,_\alpha},$  $j=1,2,3$ (see Figure \ref{SNOW}).

 When  $\alpha=3,$  we have that the the classical Koch snowflake  $\Omega_\mathsf{K}=\Omega_3.$  We point out that this set can be constructed also in the following way.
 Start with the equilateral triangle $T_1.$ Consider one of its sides $I$ and the equilateral triangle one of whose sides is the middle one third of $I$ and whose interior does not intersect $T_1.$ There are three such triangles; let $T_2$ be the closure of the union of these three triangles and $T_1.$
Then wee proceed inductively. Suppose $I$ is one of the line segments in $\partial T_j$ and consider the equilateral triangle one of whose sides is the middle one third of $I$ and whose interior does not intersect $T_j.$  Let $T_{j +1}$ be the closure of the union of all such triangles and $T_j.$ 
Then the  snowflake $\Omega_\mathsf{K}$ is the interior of the closure of the union of all triangles constructed in all inductive steps.

Now we recall  the construction of  a variant of the  Koch snowflake which can be
obtained from the snowflake $\Omega_\mathsf{K}$ as follows (\cite{BCM}) (see Figure \ref{SNOWTRAP}).

 \begin{figure}
         \centering
         \includegraphics[width=0.25\linewidth]{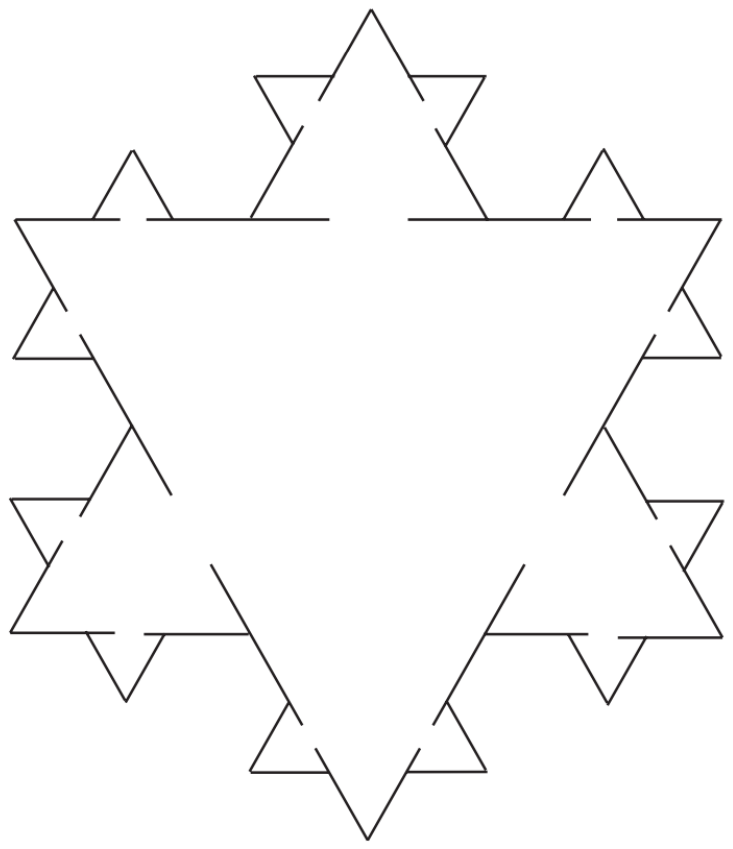}
         \caption{A variant of the Koch snowflake }
         \label{SNOWTRAP}
     \end{figure}

 Fix a function $w : (0, \infty) \to (0, \infty)$
with $w (a) \leq a$ for all $a.$ Consider any two triangles in the above construction whose
boundaries have a common part $I$ with length $a > 0.$  Let $I'$ be I with the middle
$w(a)$-portion removed, i.e., if  $I$ has endpoints $x$ and $y$ then $I'$ is the union of two closed line segments, the first with endpoints $x$ and $x +\frac{ a-w (a)}2\frac{ y-x} a,$ and the second with endpoints $y$ and   $y -\frac{ a-w(a)}2\frac{ y-x} a.$ 

Let $\Omega_\mathsf{M}$ be $\Omega_\mathsf{K}$ minus all sets of the form
$I' $. The point of the construction is that the passage from a smaller
 triangle to a bigger triangle is blocked in $\Omega_\mathsf{M}$ by a wall with a small opening $w(a).$ 
 Trivially if $w(a)=a=\frac13$ we obtain the usual Koch snowflake $\Omega_\mathsf{K}$. In a similar manner, it is possible to construct   variants of the  snowflake type fractals that we denote again by  $\Omega_\mathsf{M}$ by considering the self-similar set  $K_{\alpha,w}$ defined as the
unique closed bounded set $K_\alpha$ which is {\it invariant\/} with respect to the family $\Psi^{\alpha,w}$ where $\Psi^{\alpha,w}=\{\psi^\alpha_1,\psi^\alpha_2, \psi^\alpha_3,\psi^\alpha_4, \psi^\alpha_5,\psi^\alpha_6 \}$ with $\psi^\alpha_i$ $i=1,\ldots,4$ defined as before and 
$$\psi^\alpha_5(z)=\frac{1}{\alpha}+\beta \Re z ,\,\quad\quad\quad\quad\quad\quad\quad\quad\quad\quad\quad\quad
\psi^\alpha_6(z)=  1-\frac{1}{\alpha}-\beta+ \beta \Re z   $$
with  $$\beta= \frac12- \frac{1}{\alpha}   -\frac{w(1- \frac{2}{\alpha} )}{2};$$
in this case  $a=1- \frac{2}{\alpha}$

Now we study the trap condition   for the set $\Omega_\mathsf{M}$ 
with a  particular choice of the function  $w$  (see \cite{BCM} and \cite{BBC}).

Further on we write $w(a)= w(\gamma, a)$ with $\gamma \in \mathbb{R}$.

\begin{theorem} \label{KM}
Suppose that $w (\gamma, a) = exp(-a^{-\gamma})$. Then,
\begin{itemize}
\item For $\gamma < 2$, $\Omega_\mathsf{M}$ is not a trap domain for $X^+$;
\item For $\gamma \geq 2$, $\Omega_\mathsf{M}$ is  a trap domain for $X^+$.
\end{itemize}
\end{theorem}

\begin{proof}

 We consider a prime end $\xi$ in $\Omega_\mathsf{M}$ which is accessible  by going through an infinite sequence of  triangles  in the construction of the fractal domain as when  the prime ends correspond to boundary points accessible via a finite sequence of triangles  is easier  (see, for the definition of prime ends,  \cite{P}).
  We use the following  characterization given in Section 2.1  of \cite{BCM}: a simply connected  planar domain $D$   with finite 2-dimensional Lebesgue measure $D<\infty$  is a non-trap domain
if and only if there is a constant $\varepsilon  > 0$ such that for each prime end $\xi\in \partial D $ there is a
system of curves $\{\gamma_n\} \cup \sigma$ dividing $D$ into hyperbolic blocks with parameter $\varepsilon $  such that $$\sup_{\xi}\sum_{n} n|D_n|\leq \frac1{\varepsilon}$$
where $|D_n|$ denotes  2-dimensional Lebesgue measure of  the region $D_n$  contained between the curves $\gamma_n$ and $\gamma_{n+1}.$  
Let  $T_1$ and $T_2$   be two adjacent triangles  in this construction sequence, where $a$ denotes the side length of the smaller triangle.  The  opening between these triangles is defined by   $w(a)=exp(-a^{-\gamma})$ and 
  let $y$ be the center of  the opening. We define an index   $m$ such that  $exp(-a^{-\gamma} )\leq 2^{-m}\leq a/8.$ Let  $\gamma_m^1 =\{z\in T_1 :|y-z|=2^{-m}\},$  $\gamma_m^2 =\{z\in T_1 :|y-z|=2^{-m}\}$ and 
    $\sigma$ be the polygonal line with vertices at the center of $\Omega_\mathsf{M}$ and consecutive centers of openings between the triangles in the sequence leading to $\xi.$
   Then the system of curves given by the union of $\sigma$ and all curves $\gamma_m^1$  and $\gamma_m^2$
 corresponding to all pairs of adjacent triangles in the sequence, divides the domain into hyperbolic blocks. 
 Consider the sets $D_n$ whose boundary contains $\gamma_m^1$  or $\gamma_m^2$ corresponding to a triangle
with side length $a. $
The area of this set $D_n$ is of order  $c_1a^2$ and the number of such 
domains $D_n$ corresponding to a single triangle is of order by $c_2 a^{-\gamma}.$ Hence, the portion of   $\sum n| D_n|$ corresponding to the triangle with side length $a$ is of order  $c_1a^2c_2 a^{-\gamma}.$
As the sequence of triangle diameters  along $\sigma$ is     $a_k=a^k$  with $a<1$ we obtain   $$\sum n|D_n| \simeq  c_3  \sum  (a^{2-\gamma})^k $$ is finite if $\gamma < 2$ and it diverges to $\infty$  for $\gamma \geq 2. $

\end{proof}

 \begin{figure}
         \centering
         \includegraphics[width=0.25\linewidth]{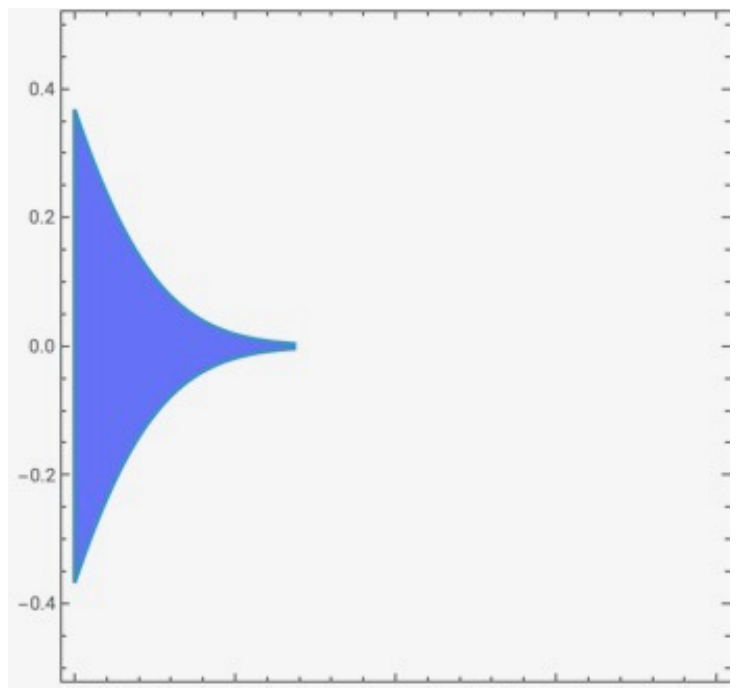}
         \caption{A Horn domain  with  $f=e^{-r}$  }
         \label{HORN}
     \end{figure}

({\it Horn domain}) A further example of trap domain is the horn domain $\Omega_\mathsf{H}.$

We recall that given 
$ f: [1,\infty)\to(0,\infty)$  a Lipschitz function,  the corresponding horn
domain is defined by
$$\Omega_\mathsf{H}= \left\{x = (x, y): \, | y | \leq f(x) \hbox{ and }  x>1 \right\}. $$

By Proposition 2.11 in \cite{BCM} a horn domain $\Omega_\mathsf{H}$ is a trap domain if and only if  $$ \int_1^\infty  \int_1^x  \frac1{f(y)}dy ) f(x)dx=\infty.$$ 
Then, for example for $f=e^{-r^b},$ we have that the corresponding horn domain is trap if and only if $b\leqslant 2$ (see Figure \ref{HORN}).

\section{Time changed diffusions}

For the Markov process $X = \{X_t,\, t \geq 0\}$ on $D$ with generator $(A, D(A))$ we denote by $\zeta$ its lifetime, that is $\zeta := \inf\{t >0\,:\, X_t \notin D\}$ meaning that $X_\zeta$ belongs to the cemetery point. Let $\tau$ be a random time and denote by $X^\tau := X \circ \tau$ the process $X$ time-changed by $\tau$. It is well-known that $X^\tau$ is Markovian only for a Markovian time change $\tau$, otherwise from a Markov process we obtain a non-Markov process. Denote by $\zeta^\tau$ the lifetime of $X^\tau$. We know that $X^\tau$ can be regarded as a delayed or rushed process with
\begin{align*}
\mathbf{P}_x(\zeta >t) < \mathbf{P}_x(\zeta^\tau > t) \leq \frac{1}{t} \mathbf{E}_x[\zeta^\tau] \;\; \textrm{or} \;\; \mathbf{P}_x(\zeta^\tau>t) < \mathbf{P}_x(\zeta > t) \leq \frac{1}{t} \mathbf{E}_x[\zeta^+].
\end{align*}
In particular, we have the following characterization.
\begin{definition}
\label{defDelRus}
(See \cite{CapDovDelRus}). For the process $X^+$ on $D$:
\begin{itemize}
\item[-] We say that $X$ is delayed by $\tau$ if $\mathbf{E}_x[\zeta^\tau] > \mathbf{E}_x[\zeta]$, $\forall \, x \in D$.
\item[-] We say that $X$ is rushed by $\tau$ if $\mathbf{E}_x[\zeta^\tau] < \mathbf{E}_x[\zeta]$, $\forall \, x \in D$.
\end{itemize}
Otherwise, we say that $X$ runs with its velocity.
\end{definition}
Similar arguments apply when replacing lifetimes with exit times, as discussed below. We now provide an overview of the context addressed hereafter.

\subsection{Non-local Cauchy problems}

For $\lambda>0$ we introduce the Bernstein function 
\begin{equation}
\Phi(\lambda) = \int_0^\infty \left( 1 - e^{ - \lambda z} \right) \Pi(dz) \quad \textrm{with} \quad \frac{\Phi(\lambda)}{\lambda} = \int_0^\infty e^{-\lambda z} \overline{\Pi}(z)dz
\label{Symbol}
\end{equation} 
where $\Pi$ on $(0, \infty)$ with $\int_0^\infty (1 \wedge z) \Pi(dz) < \infty$ is the associated L\'evy measure and $\overline{\Pi}$ is the so called tail of $\Pi$ given by $\overline{\Pi}(z) = \Pi((z, \infty))$. For details, see the books  \cite{Ber96, BerBook}. The symbol $\Phi$ can be associated with the Laplace exponent of a subordinator $H$, that is $\mathbf{E}_0[\exp( - \lambda H_t)] = \exp(- t \Phi(\lambda))$. We do not consider step-processes with $\Pi((0, \infty)) < \infty$ and therefore we focus only on strictly increasing subordinators with infinite measures. Thus, the inverse process $L$ turns out to be a continuous process with non-decreasing paths. 

Let us consider the inverse process $L_t = \inf \{s \geq 0\,:\, H_s >t \}$ and define the time-changed process $X^L = \{X^L_t,\, t \geq 0\}$. It is clear that the convolution type operator
\begin{align*}
D^\Phi_t u(t) = \int_0^t u^\prime(s) \overline{\Pi}(t-s)ds \quad (\textrm{for which } D^\Phi_t u = \dot{u} \textrm{ iff } \Phi(\lambda)=\lambda )
\end{align*}
introduces the non-local (Initial Value or) Cauchy problem $D^\Phi_t u = A u$ with $u_0=f \in D(A)$ whose solution has the probabilistic representation  
\begin{align}
\label{sol-time-frac-problem}
u(t,x) = \mathbf{E}_x[f(X^{L}_t), t < \zeta^{L}]
\end{align}
where $\zeta^{L}$ is the lifetime of $X^L$. We refer to \cite{CapDovVIA} and the references therein. Here we only underline that $D^\Phi_t$ has alternative forms investigated in \cite{chen, kochubei, toaldo}. The lifetime $\zeta^L$ has a direct connection with the following inequality, for $p \in [1, \infty)$,
\begin{align}
\label{YoungSymb}
\int_0^\infty |D^\Phi_t u |^p dt \leq \left( \int_0^\infty |u^\prime |^p dt \right) \left( \lim_{\lambda \downarrow 0} \frac{\Phi(\lambda)}{\lambda} \right)^p \quad \textrm{with} \quad \lim_{\lambda \downarrow 0} \frac{\Phi(\lambda)}{\lambda} = \Phi^\prime(0)
\end{align}
where $\Phi^\prime(0)$ is finite only in some cases (see \cite{CapDovVIA}). Our discussion is mainly concerned with the case where the functions $u$, $D^\Phi_t u$ are of exponential order such that
\begin{align*}
\mathcal{L}[D^\Phi_t u](\lambda): = \int_0^\infty e^{-\lambda t} D^\Phi_t u\, dt, \quad \mathcal{L}[u](\lambda):= \int_0^\infty e^{-\lambda t} u\, dt
\end{align*}
exist and, for $\omega\geq 0$, 
\begin{align*}
\mathcal{L}[D^\Phi_t u](\lambda) = \frac{\Phi(\lambda)}{\lambda} \left( \lambda \mathcal{L}[u](\lambda) - u_0 \right), \quad \lambda > \omega.
\end{align*}

We now provide the following result focusing on $\tau=L$, the inverse to an $\alpha$-stable subordinator. 
\begin{lemma}
Let $L$ be the inverse of the $\alpha$-stable subordinator $H$. 
\begin{itemize}
\item[i)] Assume $0 < \beta < 1/\alpha$, then
\begin{align}
    \mathbf{E}_0[(L_t)^\beta] = \frac{\Gamma(\beta + 1)}{\Gamma(\beta \alpha + 1)} t^{\alpha \beta}, \quad t>0.
    \label{meanLstable}
\end{align}
\item[ii)] Assume $0< \beta < \alpha$ and $\mathbf{E}_x[\zeta^{\beta/\alpha}]< \infty$. Then, 
    \begin{align}
\mathbf{P}_x(\zeta^L > t) \leq \left( \frac{\Gamma(1-\beta/\alpha)}{\Gamma(1-\beta)} \, \mathbf{E}_x[\zeta^{\xi/\alpha}] \right) t^{-\beta}, \quad x \in D, \; t >0.
\end{align}
\end{itemize}
\end{lemma}

\begin{proof}
The point $i)$ comes from standard calculation (see for example \cite[formula 4.9]{ECPmd}). We focus on $ii)$. From \cite[Theorem 3.1]{CapDovDelRus} we have that $E[\zeta^L] = \mathbf{E}[H_{\zeta}]$ and
\begin{align*}
\mathbf{P}_x(\zeta^L > t) \leq \frac{1}{t^\beta} \mathbf{E}_x[(H\circ \zeta )^\beta], \quad \beta < \alpha
\end{align*}
from which, by taking into account (\cite[formula 4.7]{ECPmd}))
\begin{align}
    \mathbf{E}[(H_t)^\beta] = \frac{\Gamma(1-\beta/\alpha)}{\Gamma(1-\beta)} t^{\beta/\alpha}, \quad \beta < \alpha
    \label{momH}
\end{align}
we write
\begin{align*}
\mathbf{P}_x(\zeta^L > t) \leq \frac{\Gamma(1-\beta/\alpha)}{\Gamma(1-\beta)} \, \mathbf{E}_x[\zeta^{\beta/\alpha}] \, t^{-\beta}, \quad x \in D.
\end{align*}    
\end{proof}

\begin{remark}
We notice that, for the process $X^L$ with general $\Phi$, we have the mean squared displacement 
\begin{align}
\mathbf{E}_x[(X^L_t)^2] = \mathbf{E}_x[(X_1)^2] \, \mathbf{E}_0[L_t] \sim 1/\Phi(1/t)
\end{align}
(from Proposition 1, page 74 of \cite{Ber96}) which agrees with \eqref{meanLstable}. Moreover, 
\begin{align}
\mathbf{E}_x[T_{B}(X^L)] = \mathbf{E}_x[T_{B}(X^+)]\, \mathbf{E}[H_1] \sim \Phi^\prime(0).
\end{align}

\end{remark}

\subsection{Non-local dynamic boundary value problems}

We study sticky Brownian motions (Sticky BM). We say that $\bar{X}$ is a sticky Brownian motion on $\bar{D}:= D \cup \partial D$ if we are able to identify $D_0 \subset \bar{D}$ such that the set $\{t\,:\, \bar{X}_t \in D_0\}$ has positive Lebesgue measure. Thus, we focus on occupation times. \\

We first discuss the sticky behaviour in dimension $d=1$, for metric graphs and for bounded domains of $\mathbb{R}^d$ with $d>1$.

\paragraph{Sticky BM on the real line}
The sticky Brownian motion $\bar{X}$ on $\mathbb{R}$ has been first investigated by Feller in the 1950’s and Itô and McKean in the 1960’s. If $\{0\}$ is a sticky point for $\bar{X}$, then the process $\bar{X}$ behaves like a Brownian motion on $\mathbb{R} \setminus \{0\}$. For the stochastic equation
\begin{align}
\bar{X}_t = x + \int_0^t \mathbf{1}_{(\bar{X}_s \neq 0)} dW_s
\label{SDEstickyBass}
\end{align}
where $W$ is a one-dimensional Brownian motion, Bass (see \cite{BasSticky}) proved weak uniqueness (uniqueness in law) among continuous local martingales with speed measure like  
\begin{align}
m(dx) = dx + (\eta/\sigma) \delta_0(dx)
\label{mMeasure}
\end{align}
where $\eta/\sigma$ is a positive constant (it will be clear later why we write such a constant in terms of $\eta>0$ and $\sigma>0$). In \cite{BasSticky}, it has been proved that for each constant $\eta/\sigma>0$, no strong solution to \eqref{SDEstickyBass} among the class of continuous martingales with speed measure \eqref{mMeasure} exists. Thus, given $W$, we are not able to find a continuous martingale $\bar{X}$ with speed measure \eqref{mMeasure} satisfying \eqref{SDEstickyBass} such that $\bar{X}$ is adapted to the filtration of $W$. As a consequence of this fact, some approximations to the solution of \eqref{SDEstickyBass} do not converge in probability, although they do
converge weakly.

Let us denote by $\gamma_t(y)$ the local time of $X$ on $\mathbb{R}$ at $y \in \mathbb{R}$. Then, write $\gamma_t = \gamma_t(0)$. It is well known that $\bar{X}$ can be obtained via time change as $X \circ V^{-1}_t$ where 
\begin{align*}
V_t = \int_0^\infty \gamma_t(x) m(dx) = t + (\eta/\sigma) \gamma_t, \quad t\geq 0
\end{align*}
and (almost surely) $V \circ V^{-1}_t = t$ holds. In particular, for $f \in C_b(\mathbb{R})$, we have
\begin{align}
\mathbf{E}_x[f(\bar{X}_t)] = \mathbf{E}_x[f(X\circ V^{-1}_t)], \quad t \geq 0,\, x \in \mathbb{R}.
\label{repSemigX}
\end{align}
Formula \eqref{repSemigX} provide the probabilistic representation of the Feller-Wentzell semigroup. Thus, the weak formulation here is meant in terms of PDEs. In particular, the heat equation is equipped with a dynamic boundary condition.\\

Recently, non-local dynamic boundary conditions have been considered (see \cite{DovSticky1}). The (classical) solution $u \in C((0, \infty) \times [0, \infty)) \cap D_L$ to
\begin{align}
\label{FBVP-realLine}
\left\lbrace
\begin{array}{ll}
\displaystyle \dot{u}(t,x) = u^{\prime \prime}(t,x), \quad t>0, \, x \in (0,\infty),\\
\\
\displaystyle \eta D^\Phi_t u(t,0) = \sigma u^\prime(t,0) - c\, u(t, 0), \quad t>0,\\
\\
\displaystyle u(0,x) = f(x), \quad x \geq 0, \quad f \in C_b[0, \infty),
\end{array}
\right.
\end{align}
(where $u \in D_L$ ensures existence of $D^\Phi_t u$) has the probabilistic representation
\begin{align*}
u(t,x) = \mathbf{E}_x [f(\bar{X}_t)] = \mathbf{E}_x\left[f(X^+ \circ V^{-1}_t) \exp \left( -(c/\sigma)\, \gamma^+ \circ V^{-1}_t \right) \right]
\end{align*}
where $V^{-1}_t$ is the inverse of the process 
$$V_t = t + H \circ (\eta/ \sigma) \gamma^+_t,$$
where $H$ is a subordinator with symbol $\Phi$ independent from $X^+$ and $\gamma^+=\gamma^+_t(0)$ is the local time of $X^+$. The process $\bar{X}$ on $[0, \infty)$ behaves like a Brownian motion only on $(0, \infty)$ and we can still consider \eqref{SDEstickyBass}. However, the process $\bar{X}$ now is Markov on $[0, \infty)$ only for $\Phi(\lambda)=\lambda$. For all possible choices of $\Phi$, the zero set $\{t\,:\, \bar{X}_t=0\}$ has positive Lebesgue measure. Despite the fact $\bar{X}$ moves along the path of $X^+$, we are able to identify a sequence $\{\bar{e}_i\}_i$ of holding times (at $\{0\}$) for $\bar{X}$ of independent and identically distributed random variables, in particular $\bar{e}_i$ equals in law $H_\chi$ with $\mathbf{P}(\chi > z) = e^{-(\sigma/\eta) z}$ for all $i$. As an example we recall the case $\Phi(\lambda)=\lambda^\alpha$ in which the holding times are Mittag-Leffler random variables. Assume $c=0$ and observe that, for $B=B(0, \varrho)$ and $\varrho>0$, 
\begin{align*}
    \mathbf{E}_x[f(\bar{X}), \, t < T_B(\bar{X})] =  \mathbf{E}_x[f(X^+ \circ V^{-1}_t), \, t < V \circ T_B(X^+)], \quad x \in B
\end{align*}
leads to $\mathbf{E}_x[T_B(\bar{X})] = \mathbf{E}_x[T_B(X^+)] + (\eta/\sigma) \Phi^\prime(0)\, (\varrho-x)$ where $\Phi^\prime(0) = \mathbf{E}[H_1]$ and $\gamma^+ \circ T_B$ is an exponential r.v. with $\mathbf{E}_x[\gamma^+ \circ T_B]=\varrho-x$. In the present work we mainly focus on stable subordinators for which $\Phi^\prime(0)=\infty$. We provide the following result.
\begin{theorem}
    Let $\theta\geq 0$, $\beta < \alpha$. Then,
    \begin{align}
         t^\theta \big( 1 - C\, t^{\beta (\frac{1}{2\alpha} -1)}\big) \leq \mathbf{E}_0[(V^{-1}_t)^\theta] \leq t^\theta
    \end{align}
    with
    \begin{align}
        C=C(\alpha, \beta, \theta) = 4^\frac{\beta}{2\alpha} \frac{\Gamma(\frac{\beta}{2\alpha} + \frac{1}{2})}{\Gamma(1/2)} \theta \frac{\Gamma(\theta + \frac{\beta}{2\alpha}) \Gamma(1-\frac{\beta}{\alpha})}{\Gamma(\theta + \frac{\beta}{2\alpha} - \beta +1 )} (\eta/\sigma)^{\beta/\alpha}.
    \end{align}
    \label{thm:momentsInvVbar}
\end{theorem}
\begin{proof}
    Since $V_t \geq t$ we get that $V^{-1}_t \leq t$. Moreover, 
\begin{align*}
    \mathbf{E}_0[(V^{-1}_t)^\theta] 
    = & \int_0^t \theta s^{\theta -1} \mathbf{P}_0(V^{-1}_t > s) ds\\
    = & t^\theta - \int_0^t \theta s^{\theta -1} \mathbf{P}_0(V^{-1}_t \leq s) ds\\
    = & t^\theta - \int_0^t \theta s^{\theta -1} \mathbf{P}_0(t \leq V_s) ds\\
    = & t^\theta - \int_0^t \theta s^{\theta -1} \mathbf{P}_0(H \circ (\eta/\sigma) \gamma^+_s  \geq t-s) ds\\
    \geq & t^\theta - \mathbf{E}_0[(H_{(\eta/\sigma)})^\beta] \int_0^t \theta s^{\theta -1} \frac{\mathbf{E}_0[(\gamma^+_s)^{\beta/\alpha}]}{(t-s)^\beta} ds\\
    = & t^\theta - \mathbf{E}_0[(H_{(\eta/\sigma)})^\beta] 4^\frac{\beta}{2\alpha} \frac{\Gamma(\frac{\beta}{2\alpha} + \frac{1}{2})}{\Gamma(1/2)} \int_0^t \theta s^{\theta -1} \frac{s^\frac{\beta}{2\alpha}}{(t-s)^\beta} ds\\
    = & t^\theta - \mathbf{E}_0[(H_{(\eta/\sigma)})^\beta] 4^\frac{\beta}{2\alpha} \frac{\Gamma(\frac{\beta}{2\alpha} + \frac{1}{2})}{\Gamma(1/2)} \theta \frac{\Gamma(\theta + \frac{\beta}{2\alpha}) \Gamma(1-\beta)}{\Gamma(\theta + \frac{\beta}{2\alpha} - \beta +1 )} t^{\theta + \frac{\beta}{2\alpha} - \beta}\\
    = & t^\theta - 4^\frac{\beta}{2\alpha} \frac{\Gamma(\frac{\beta}{2\alpha} + \frac{1}{2})}{\Gamma(1/2)} \theta \frac{\Gamma(\theta + \frac{\beta}{2\alpha}) \Gamma(1-\frac{\beta}{\alpha})}{\Gamma(\theta + \frac{\beta}{2\alpha} - \beta +1 )} (\eta/\sigma)^{\beta/\alpha}  t^{\theta + \frac{\beta}{2\alpha} - \beta}\\
    = & t^\theta \big( 1 - C\, t^{\beta (\frac{1}{2\alpha} -1)}\big).
\end{align*}
\end{proof}

\paragraph{Sticky BM on metric graphs}
Let us define $\bar{X}^\dagger = \{\bar{X}^\dagger_t\}_{t\geq 0}$ as the process $\bar{X}$ on $[0, \ell)$ and killed at $\ell>0$. Write $\bar{\tau}_\ell = \inf\{t\,:\, \bar{X}_t=\ell\}$ and $\tau^+_\ell = \inf\{t\,:\, X^+_t=\ell\}$. Then, for $f \in C[0,\ell)$,
\begin{align*}
\mathbf{E}_x[f(\bar{X}^\dagger_t)] = \mathbf{E}_x[f(\bar{X}_t), t < \bar{\tau}_\ell] = \mathbf{E}_x[f(X^+ \circ V^{-1}_t), V^{-1}_t < \tau^+_\ell], \quad t\geq 0,\; x \in [0, \ell).
\end{align*}    
The process $\bar{X}^\dagger$ can be considered as the killed version of the process introduced \cite{BonDovStar}, that is a sticky Brownian motion on star graphs. In particular, given $\mathcal{E} := \{ \varepsilon=[0,\ell)  \}$ as the collection of rays given by the bounded intervals of the real line, we define the star graph $\mathsf{S}$ as the quotient space $\mathsf{S} = \mathcal{E}/\sim$, i.e., we identify the starting points on all edges and in $\mathcal{E}$ the origin $0 \equiv (\cdot, 0)$ is the unique point that belongs to all the rays. Such a point is identified as the vertex $\mathsf{v}$ of $\mathsf{S}$. We identify a point $\mathsf{x} \in \mathsf{S}$ as $\mathsf{x}=(\varepsilon, x)$ for the edge $\varepsilon$ and the distance $x$ from the star vertex. A sticky Brownian motion $\mathsf{Y}$ on $\mathsf{S}$ is therefore defined as the couple $(\varepsilon, \bar{X}^\dagger_t)$ on the edge $\varepsilon \in \mathcal{E}$.

The process $\bar{X}$ is Markov only away from $\{0\}$, thus $\mathsf{Y}$ is Markov on $\mathsf{S} \setminus \{\mathsf{v}\}$. For the sake of clarity we write the problem for $\mathsf{Y}$ on the star graph $\mathsf{S}$,
\begin{equation}
\left\lbrace
\begin{array}{ll}
\displaystyle \dot{u}(t, \mathsf{x}) = \mathsf{G} u(t, \mathsf{x}), & \mathsf{x} \in \mathsf{S} \setminus \{\mathsf{v}\}\\
\\
\displaystyle \eta \,D^\Phi_t u(t, \mathsf{v}) = \sum_{\varepsilon \in \mathcal{E}} \rho_\varepsilon\, u^\prime_\varepsilon(t, 0), & t>0, \quad \eta >0,\\  
\displaystyle u_\varepsilon(t, \ell) = 0, & t>0, \quad \varepsilon \in \mathcal{E},\\ 
\\
\displaystyle u(0, \mathsf{x}) = f(\mathsf{x}), & \mathsf{x} \in \mathsf{S}, \quad f \in C(\mathsf{S}).
\end{array}
\right.
\label{NLBVPGraphINTRO}
\end{equation}
where $\mathsf{G}$ is the Laplacian on $\mathsf{S}$ and $u_\varepsilon, u^\prime_\varepsilon$ are the projection of $u, u^\prime$ on the edge $\varepsilon \in \mathcal{E}$. For a rigorous setting we refer to \cite{BonDovStar} and \cite{BonColDovPag}. Here we provide new results for the local time of $\mathsf{Y}$ at the star vertex $\mathsf{v}$. Recall that $\gamma^+_t$ at $\{0\}$ up to the first hitting time $\tau_\ell$ is an exponential r.v. with $\mathbf{E}_x[\gamma^+ \circ \tau_\ell] = (\ell - x)$. Thus,
\begin{align*}
\mathbf{E}_x[\bar{\tau}_\ell] = \mathbf{E}_x[\tau_\ell] + (\eta/\sigma) \Phi^\prime(0) \mathbf{E}_x[\gamma^+ \circ \tau_\ell] = \frac{\ell^2 - x^2}{2} + (\eta/\sigma) \Phi^\prime(0) (\ell - x)
\end{align*}
and
\begin{align*}
\mathbf{E}_{\mathsf{v}} [\inf\{t\,:\, \mathsf{Y}_t \notin \mathsf{S} \}] = \frac{\ell^2}{2} + (\eta/\sigma) \Phi^\prime(0) \ell
\end{align*}
is the mean exit time from a star graph. With no abuse of notation we write 
\begin{align*}
\gamma_s(\mathsf{v}) = \int_0^t \mathbf{1}_{\{\mathsf{v}\}}(\mathsf{Y}_s) ds 
\end{align*}
and underline that the set $\{t\,:\, \mathsf{Y}_t = \mathsf{v}\}$ has positive Lebesgue measure.

\begin{theorem}
For $\alpha < 1/2$, as $t\to 0$, we have that 
\begin{align}
\mathbf{E}_{\mathsf{m}} \left[ \frac{1}{t} \int_0^t f(\mathsf{Y}_s) d\gamma_s(\mathsf{v}) \right] \to \int_{\mathsf{S}} f(\mathsf{x}) \mu(d\mathsf{x})
\label{RevMeasure}
\end{align}
where
\begin{align*}
\mu(d\mathsf{x}) = \sum_{\varepsilon \in \mathcal{E}}  \rho_\varepsilon\, \mu_\varepsilon(dx)
\end{align*}
is the Revuz measure of the local time $\gamma(\mathsf{v}) = \gamma(\cdot, 0)$ at the star vertex $\mathsf{v}$ of the process $\mathsf{Y}$ on $\mathsf{S}$ whereas $\mu_\varepsilon(dx) = \mu(\varepsilon, dx)$ can be associated with the local time $\gamma(\mathsf{v}) = \gamma(\varepsilon, 0)$ at $\mathsf{v}$ of $\mathsf{Y}$ on the edge $\varepsilon$.
\end{theorem}
\begin{proof}
Recall that $\mathsf{Y}_0 = (\varepsilon, \bar{X}^\dagger_0)$ and $\mathsf{S} \ni \mathsf{x} = (\varepsilon, x)$ with $\varepsilon \in \mathcal{E}$, $x \in [0, \ell)$. Moreover, the zero set of $\bar{X}$ has positive Lebesgue measure. Observe that $\mathsf{v}=(\cdot, 0),$  
\begin{align*}
\mathbf{E}_{\mathsf{m}} \left[ \frac{1}{t} \int_0^t f(\mathsf{Y}_s) d\gamma_s(\mathsf{v}) \right]
= & \int_{\mathsf{S}}\frac{1}{t} \int_0^t  \mathbf{E}_{\mathsf{x}}[f(\mathsf{Y}_s) \mathsf{1}_{\{\mathsf{v}\}}(\mathsf{Y}_s)] \, ds\, \mathsf{m}(d\mathsf{x})\\
 &( \textrm{ recall that } \{t\,:\, \bar{X}_t=0\} \textrm{ has positive Lebesgue measure})\\
= & \sum_{\varepsilon \in \mathcal{E} }  \rho_\varepsilon \int_0^\ell  \frac{1}{t} \int_0^t    \mathbf{E}_x[f_\varepsilon(\bar{X}^\dagger_s) \mathbf{1}_{\{0\}}(\bar{X}^\dagger_s)] \, ds\, dx\\
& (\textrm{ recall that } \bar{X} = X \circ V^{-1} \textrm{ and } \bar{\gamma} = \gamma^+ \circ V^{-1})\\
= & \sum_{\varepsilon \in \mathcal{E} }  \rho_\varepsilon \int_0^\ell  \mathbf{E}_x \left[ \frac{1}{t} \int_0^{t \wedge \bar{\tau}_\ell} f_\varepsilon(\bar{X}_s) d\bar{\gamma}_s(0) \right]\, dx\\
= & \sum_{\varepsilon \in \mathcal{E} }  \rho_\varepsilon \int_0^\ell  \mathbf{E}_x \left[ \frac{V^{-1}_{t \wedge \bar{\tau}_\ell}}{t} \frac{1}{V^{-1}_{t \wedge \bar{\tau}_\ell}} \int_0^{V^{-1}_{t \wedge \bar{\tau}_\ell}} f_\varepsilon(X^+_s) d\gamma^+_s(0) \right]\, dx.
\end{align*}
Now observe that $V^{-1}_t \leq t$ a.s. and $V^{-1}_t/t \to 1$ as $t \to 0$ if $\alpha < 1/2$. Indeed, (see Theorem \ref{thm:momentsInvVbar})
    \begin{align}
         \big( 1 - C\, t^{\beta (\frac{1}{2\alpha} -1)}\big) \leq \mathbf{E}_0[(V^{-1}_t / t)^\theta] \leq 1
         \label{fMomInvVbarINproof}
    \end{align}
with $\theta>0$, $\beta < \alpha$. 

Finally, we underline that
\begin{align}
0 \leq \bar{V}^{-1}_t \leq t \textrm{ a.s. implies }  \mathbf{P}(\lim_{t \to 0} \bar{V}^{-1}_t = 0)=1
\end{align}
and, as $t \to 0$,
\begin{align*}
\mathbf{E}_x \left[ \frac{1}{\bar{V_t}} \int_0^{\bar{V}_t} f_\varepsilon(X^+_s) d\gamma^+_s(0)\right] \to \int_0^\ell f_\varepsilon(x) \mu^+(dx)
\end{align*}
where $\mu^+$ is the Revuz measure associated with $\gamma^+$ at $\{0\}$. That is, the local time $\gamma^+(0)$ of $X^+$ on $[0, \ell)$. Thus, as $t \downarrow 0$, we get \eqref{RevMeasure}.
\end{proof}

\paragraph{Sticky BM on domains in the class $\mathcal{D}$}

We extend the result in \cite{DovSticky2} for smooth domains to  not smooth domains.

A real-valued function $f$ defined on $A \subset  \R^d$ is called Lipschitz with constant $\lambda < \infty $ if
$|f(x)- f(y)| \leqslant  \lambda |x - y|$ for all $x, y \in A.$  A domain D is called Lipschitz if there exist $r > 0$ and and $\lambda  < \infty$ such that for every $x \in  \partial D,$ the set $D \cap  B(x, r)$ is the graph of a Lipschitz
function with constant $\lambda$ in some orthonormal coordinate system. We call $(\lambda , r)$ the Lipschitz
characteristics of $D.$ 

We recall  the definition of the class $\mathcal{D}$  (see  \cite{BBC}).

\begin{definition}
We will say that a domain $D$ belongs to class $\mathcal{D}$ if there exists an increasing
sequence of domains $D_n \subset D$ with the following properties.\begin{itemize}
\item (i) Each $D_n$ is a Lipschitz domain with characteristics $(\lambda, r_n)$ and $\cup_{n=1}^\infty D_n = D.$
\item (ii) For every $n\geq 1,$ the set $\partial D_n \cap  \partial D$ is a subset of the relative interior of $\partial D_{n+1} \cap  \partial D.$
\item (iii) $\sup_{n\geqslant 1} |\partial D_n| < \infty$  and $\lim_{n\to\infty} |\partial D_n \setminus  \partial D| = 0.$
\end{itemize}
\end{definition}

The set $\partial_L D = \cup_n \partial D_n \cap \partial D$ will be called the Lipschitz part of $\partial D.$
Every bounded Lipschitz domain is in the class  $\mathcal{D}$.
For domains $D\in   \mathcal{D}$ which are not Lipschitz see Examples 3.4, 3.6, 4.13, 4.15 and 4.16 in \cite{BBC}.

Now we recall the construction of a reflecting Brownian motion on a non-smooth domain $D$ (see, for example, \cite{BBC}).

\begin{theorem}
If $D \in  \mathcal{D}$, then reflecting Brownian motion $X^+$ in $D$ starting from $x \in D \cap \partial_L D $  has a semimartingale decomposition 
\begin{align}
X^+_t = x + W_t + N_t, \quad t\geq 0
\end{align}
where $W_t$ is a d-dimensional Brownian motion, $N_t =\int_0^t n(X_s)d\gamma^+_s$ and the local time $\gamma^+$, is a non-decreasing continuous process that does not increase when $X^+$ is not in $\partial L_D$, that is 
\begin{align*}
\int_0^\infty  1_{ (\partial_L D)^c }n(X^+_t)d\gamma^+_t=0
\end{align*}
(here $n$ is the unit inward normal vector field of $D$ on $\partial_L D)$.
The Revuz measure of $\gamma^+$ for the process $X^+$ is the surface measure on $\partial_L D$.
\end{theorem}
Now for $D \in \mathcal{D}$ we consider the problem 
\begin{align}
\left\lbrace
\begin{array}{l}
\displaystyle \frac{\partial u}{\partial t}(t,x) = \Delta u(t,x), \quad t>0, \; x \in D, \\
\\
\displaystyle \eta D^\Phi_t Tu(t,x) = - \sigma \partial_{\bf n} u(t,x) - c \, u(t,x), \quad t>0, \; x \in \partial D,\\
\\
\displaystyle u(0,x) = f(x), \quad x \in \overline{D}.
\end{array}
\right .
\label{probXbar}
\end{align}
where $\partial_{\bf n}\varphi$ is the outer normal derivative with respect to $m(dx) = \mathbf{1}_D dx + \mathbf{1}_{\partial D} d\mathcal{H}$ (the sum of the $d$-dimensional Lebesgue measure supported on the interior and the Hausdorff measure supported on the boundary). We denote by $Tu=u|_{\partial D}$ the trace function (continuous operator) from $H^1(D)$ into $L^2(\partial D, \mathcal{H})$ for $u \in H^1(D)\cap C(\overline{D})$. We will use the following space
\begin{align*}
\bar{D}_L = \bigg\{ & \varphi: (0, \infty) \times \overline{D} \to \mathbb{R} \textrm{ with } \phi= T\varphi \textrm{ such that } \\ 
 &  \dot{\phi} \in C((0, \infty) \times \partial D)  \textrm{ and }  \dot{\phi}(s, x) \kappa(t-s) \in L^1(0, t),\, \forall\, x \in \partial D,\,  t>s>0 \bigg\}
\end{align*}
and the fact that $\bar{D}_L=\bar{D}_L(D)$ will be considered below.

\begin{theorem}
The solution $u \in C((0, \infty), D) \times \bar{D}_L$ to \eqref{probXbar} has the representation
\begin{align}
\mathbf{E}_x\left[ f(\bar{X}_t) \right] = & \mathbf{E}_x\left[f(X^+ \circ \bar{V}^{-1}_t) \exp \left( - (c/\sigma)\, \gamma^+ \circ \bar{V}^{-1}_t \right) \right], \quad t>0,\; x \in \overline{D}
\label{repXbar}
\end{align}
where $\bar{V}^{-1}_t$ is the inverse to 
\begin{align}
\bar{V}_t = t + H \circ (\eta/\sigma) \gamma^+_t, \quad t\geq 0.
\end{align} 
\end{theorem}
\begin{proof}
The sticky Brownian motion on smooth domains $\Omega_\mathsf{S} \subset \mathbb{R}^d$, $d > 1$ with non-local behavior on the boundary has been investigated in \cite{DovSticky2}. Observe that this result is concerned with the probabilistic representation of $u$ with compact representation on $L^2(\overline{\Omega}_\mathsf{S},m) = L^2(\Omega_\mathsf{S}) \oplus L^2(\partial \Omega_\mathsf{S})$. Let us introduce the space $\bar{D}_L=\bar{D}_L(\Omega_\mathsf{S})$. The solution $u \in C((0, \infty), \overline{\Omega}_\mathsf{S}) \cap \bar{D}_L$ to the FBVP \eqref{probXbar} with $f \in C(\overline{\Omega}_\mathsf{S})$ has the representation \eqref{repXbar} in terms of an elastic Brownian motion on $\Omega_\mathsf{S}$ (see \cite{DovSticky2}).

Consider now the domain $D \subset \mathcal{D}$. For the Skorokhod equation in $\bar{D}$
\begin{align*}
X^+_t = x + W_t + \int_0^t n(X^+_s) d\gamma^+_s
\end{align*}
we have strong existence and path uniqueness of the solution in case the Lipschitz constant is strictly less than 1. Moreover, we have weak uniqueness in case of planar domains. Such results have been proved in \cite{BBC05} under some additional regularity conditions for the boundary local time $\gamma^+$. We are therefore justified to deal under equivalence in law. Given the couple $(X^+, \gamma^+)$ we are able to solve \eqref{probXbar} in the space $C((0, \infty), \overline{D}) \cap \bar{D}_L(D)$.
\end{proof}

\section{Trap behaviours}

Beside the characterization of the domain, it is well-known that a motion can exhibit anomalous behaviour leading to random traps as an endogenous property (it can not be charged to the geometry of the domain). Despite the fact that a particle can not change its nature, the observation of that particle can be given in an unknown medium, then we only see a motion affected by some anomalies.\\

\subsection{Characterization in terms of reflected BM}\label{RBM}

Let us consider the time change $\tau$ and the (continuous) reflecting motion $Z = \{Z_t \}_{t \geq 0}$ on $\bar{D}$, a non empty open set of $\mathbb{R}^d$ with $d\geq 1.$ 

Let $B(x,\rho)$ be the ball centered at $x \in \bar{D}$ with positive radius $\rho$ and such that $B(x,\rho) \cap \bar{D}$ and $B^c(x, \rho) \cap \bar{D}$ are not empty. We assume that $D$ is non-trap for $Z$ according with Definition \ref{defTrapd}. For the exit time $T_B(Z)$ now we need all the details and write $T_{B(x, \rho)} (Z) := \inf \{t\,:\, Z_t \notin B(x,\rho)\}$. Observe that $\forall\, \rho>0$,
\begin{align*}
\sup_{x \in D}  \mathbf{E}_x[T_{B(x,\rho)}(Z)] = \sup_{x \in D} \mathbf{E}[T_{B(x,\rho) \cap \bar{D}}(Z)\, |\, Z_0=x ].
\end{align*}

\begin{definition}
We say that $Z$ has a trap behavior in $\bar{D}$ if
\begin{align*}
\exists\, \rho >0\, :\, \sup_{x \in D}  \mathbf{E}_x[T_{B(x,\rho)}(Z)] = \infty.
\end{align*} 
\label{defTrapb}
\end{definition}

\begin{remark}
    If $Z$ equals in law $X^\tau:=X^+ \circ \tau$, then we say that $\tau$ introduces a trap behavior of $X^+$ on $D$.
\end{remark}

\begin{remark}
    If $D$ is a trap domain for a reflected Brownian motion, then the reflected Brownian motion has a trap behavior in $D$. Moreover, if we have a trap behavior in $D$, then there exists a non-empty set $\Lambda \subseteq \overline{D}$ such that the reflected Brownian motion has a trap behavior in $\Lambda$.  
\end{remark}

\paragraph{On the Koch snowflakes}

We now show that $\bar{X}$  on $\bar{D}$ exhibits a trap behaviour on $\bar{D}$.

\begin{theorem}
Let $B \subset \bar{D}$ with $x \in B$.
\begin{itemize}
\item[i)] Assume $B \cap \partial D = \emptyset$, then $\mathbf{E}_x[T_B(\bar{X})] = \mathbf{E}_x[T_B(X^+)]$;
\item[ii)] Assume $B \cap \partial D \neq \emptyset$, then $\mathbf{E}_x[T_B(\bar{X})] < \infty$ only if $\Phi^\prime(0)< \infty$.
\end{itemize}
\label{thm:TrapBehBarX}
\end{theorem}
\begin{proof}
Let us consider the ball $B \subset \bar{D}$. For $X^+$ on $\bar{D}$ with $X^+_0=x \in B$ we consider the killing time $T_B(X^+)$. Analogously, for $\bar{X}$ on $\bar{D}$ with $\bar{X}_0 =x \in B$ we consider $T_B(\bar{X})$. For $x \in B$,
\begin{align*}
\int_0^\infty \mathbf{E}_x[\mathbf{1}(X^+ \circ V^{-1}_t),\, V^{-1}_t < T_B(X^+)] dt = & \int_0^\infty \mathbf{E}_x[\mathbf{1}(\bar{X}_t),\, t< V \circ T_B(X^+)] dt \\ 
= & \int_0^\infty \mathbf{E}_x[\mathbf{1}(\bar{X}_t),\, t < \bar{T}_B(\bar{X})] dt
\end{align*}
and we get
\begin{align*}
\mathbf{E}_x[T_B(\bar{X})] = \mathbf{E}_x[V \circ T_B(X^+)].
\end{align*}
That is, 
\begin{align*}
\mathbf{E}_x[T_B(\bar{X})] = \mathbf{E}_x[T_B(X^+)] + (\eta/\sigma) \mathbf{E}_0[H_t]\, \mathbf{E}_x[\gamma^+ \circ T_B(X^+)]
\end{align*}
where $\mathbf{E}_0[H_t] = \Phi^\prime (0)$. We underlined that $\mathbf{E}_x[\gamma^+ \circ T_B(X^+)] = a(x) + b(x)$ is written in terms of
\begin{align*}
a(x) : = \mathbf{E}_x[\gamma^+ \circ T_B(X^+), T_B(X^+) < T_{D}(X^+)] = 0\\
b(x) : = \mathbf{E}_x[\gamma^+ \circ T_B(X^+), T_B(X^+) \geq T_{D}(X^+)] >0.
\end{align*}
Indeed, $\gamma^+_t >0$ a.s. only if $t \geq T_{D}(X^+)$. Thus, in case $B \cap \partial D = \emptyset$ and $\bar{X}_0=x \in B$, we simply get $\mathbf{E}_x[T_B(\bar{X})]= \mathbf{E}_x[T_B(X^+)]$. If $B \cap \partial D \neq \emptyset$, then for $\bar{X}_0=x \in B$ we have 
\begin{align*}
\mathbf{E}_x[T_B(\bar{X})] = \mathbf{E}_x[T_B(X^+)] + \Phi^\prime(0)\, b(x), \quad x \in B
\end{align*} 
which is finite only if $\Phi^\prime(0)< \infty$.

\end{proof}

Our results also provide useful application as the following theorem entails. Let us assume $c=0$. Then, $\bar{X}$ on $\overline{\Omega}_\mathsf{K}$ is a reflecting Brownian motion with non-local dynamic boundary condition. Let us consider $X^+$ on $\overline{\Omega}_\mathsf{M}$ as defined in Section \ref{sec:TrapDom}. We recall that $\alpha \in (2,4)$ is the scale factor associated with both domains $\Omega_\mathsf{K}$ and $\Omega_\mathsf{M}$.\\

\begin{theorem}
Let $\beta \in (0,1)$. Let $\Phi(\lambda) = \lambda^{\beta}$ characterize the sticky process $\bar{X}$. Let $w(\beta, 1-2/\alpha) = \exp(-(1-2/\alpha)^{-(2/\beta)})$ characterize the trap domain $\Omega_\mathsf{M}$. Then, according to Definition \ref{defTrapb}  the processes $\bar{X}$ on $\overline{\Omega}_\mathsf{K}$ and $X^+$ on $\overline{\Omega}_\mathsf{M}$ exhibit the same trap behaviour depending on $\beta$. 
\end{theorem}
\begin{proof}
From Theorem \eqref{thm:TrapBehBarX} we know that the process $\bar{X}$ exibits a trap behaviour according to Definition \ref{defTrapb}. That is, for $B \subset \overline{\Omega}_\mathsf{K}$, 
\begin{align*}
\mathbf{E}[T_B(\bar{X})\,| \, \bar{X}_0=x \in B]< \infty \quad \textrm{if} \quad B \cap \partial \Omega_\mathsf{K} = \emptyset
\end{align*}
whereas
\begin{align*}
\mathbf{E}[T_B(\bar{X})\,| \, \bar{X}_0=x \in B] = \infty \quad \textrm{if} \quad B \cap \partial \Omega_\mathsf{K} \neq \emptyset
\end{align*}
and we say that $\bar{X}$ exhibits a trap behaviour on $\overline{\Omega}_\mathsf{K}$. Observe that, for the Dirichlet ball $B_0 \subset \Omega_\mathsf{K}$ with $B_0 \cap \partial \Omega_\mathsf{K} = \emptyset$, according to Definition \ref{defTrapb}, $\bar{X}$ exhibits a trap behaviour on $\overline{\Omega}_\mathsf{K} \setminus B_0$. That is,
\begin{align*}
\exists\, \rho \, :\, \sup_{x \in \overline{\Omega}_\mathsf{K} \setminus B_0} \mathbf{E}[T_{B(x, \rho)} (\bar{X})\, | \, \bar{X}_0=x \in B(x, \rho)] = \infty. 
\end{align*}
This is true if $B(\cdot, \rho) \cap \partial \Omega_\mathsf{K} \neq \emptyset$ for some $\rho >0$. 

On the other hand, from Section \ref{sec:TrapDom} we know that $\Omega_\mathsf{M}$ is a trap domain, that is, given the Dirichlet ball $B_0 \subset \Omega_\mathsf{M}$ with $B_0 \cap \partial \Omega_\mathsf{M}= \emptyset$, we have 
\begin{align*}
\sup_{x \in \bar{\Omega}_\mathsf{M} \setminus B_0} \mathbf{E}_x[T_{B_0}(X^+)] = \infty.
\end{align*} 
Analogously, by Definition \ref{defTrapb},
\begin{align*}
\exists\, \rho \, :\, \sup_{x \in \overline{\Omega}_\mathsf{M} \setminus B_0} \mathbf{E}[T_{B(x, \rho)} (X^+)\, | \, X^+_0=x \in B(x, \rho)] = \infty.
\end{align*}
Indeed, we can chose a ball $B(\cdot, \rho) \supset B_0$ such that $B(\cdot, \rho) \cap \partial \Omega_\mathsf{M} \neq \emptyset$ for some $\rho >0$.
\end{proof}

\begin{remark}
($\beta\to 1$) We observe that $\bar{X}$ behaves like a Brownian motion under Feller-Wentzell boundary condition. The holding times of $\bar{X}$ on the boundary are exponential random variables and the boundary occupation time is finite. Indeed, $\Omega_\mathsf{M}$ with $w(1, 1-2/\alpha) = \exp(-(1-2/\alpha)^{-2}) < 1/\alpha$ still maintain portion of the walls and the process $X^+$ can spend some extra time near the boundary. However, $\Omega_\mathsf{M}$ turns out to be non-trap.
\end{remark}

\begin{remark}
    ($\beta \to 0$) In this case, due to the long jumps of $H$ the process $\bar{X}$ spends more time near the boundary (the average is infinite as well as $\beta < 1$). Concerning $X^+$ on $\Omega_\mathsf{M}$ we have that the window $w$ becomes smaller and smaller.  
\end{remark}

\paragraph{On the metric graph}

We stress the fact that the arguments above can be considered for a locally compact metric space.

 We are mainly interested in the case of metric graph $\mathsf{G}$  obtained as a collection of star graphs $\mathsf{S}$. Assume $\mathsf{G}$ is given by $(\mathcal{E}, \mathcal{V})$ where the set of vertices contains a non-empty set of trap vertices, say $\mathcal{V_\mathsf{trap}} \subset \mathcal{V}$. Thus, we still obtain a characterization in terms of trap behaviour on $\mathsf{G}$.  The trap vertices are defined as the set of vertices for which a non-local dynamic condition can be specified. For a given ball $B=B(x, \rho)$, the process $\bar{X}$ on $\mathsf{G}$ with $\bar{X}_0\in B$ satisfies 
\begin{align*}
\mathbf{E}_{(\varepsilon, x)}[T_B(\bar{X})] = \mathbf{E}_{(\varepsilon, x)}[V \circ T_B(X^+)] 
\end{align*}
that is, as discussed above, 
\begin{equation}
\mathbf{E}_\mathsf{v}[T_B(\bar{X})] \sim \frac{\rho^2}2 + (\eta/\sigma) \Phi^\prime(0)\rho  , \quad \mathsf{v}\in  \mathcal{V}_\mathsf{trap}\end{equation}
whereas in the case   $ B \cap \mathcal{V} = \emptyset$
\begin{equation}
\mathbf{E}_{(\varepsilon, x)}[T_B(\bar{X})] =\frac{\rho^2-x^2}2 \quad \varepsilon\in \mathcal{E}.\end{equation}

We underline the fact that this formulation can be related with the Fontes-Isopi-Newman model (the so-called FIN  model, see  \cite{FIN})  and the Bouchaud trap's model for which
\begin{align}
    V_t = t + \sum_j \eta_j\, \gamma_t(\mathsf{v_j}) = \int_{\mathsf{G}}\, \gamma^+_{t,z} m(dz) 
\end{align}
where
\begin{align}
    m(dz) = dz + \sum_j \eta_j \, \delta_j(dz) 
\end{align}
where $\delta_j(dz) = \delta_{\mathsf{v}_j}(dz)$ is the Delta measure and $\{\mathsf{v}_j\}_j = :\mathcal{V}_\mathsf{trap} \subset \mathcal{V}$ is the set of sticky points (vertices) in $\mathsf{G}$.

\subsection{Characterization in terms of killed BM}
We now give a characterization in terms of killed BM.

In the following,  we write $h(s) \simeq f(s)$ if there exist constants $c_1,c_2 > 0$ such that $c_1 f(s) \leq h(s) \leq c_2f (s)$ for the specified range of the argument s.

\begin{definition}
Let $Z^\dagger =\{Z_t, \, t < T_D(Z)\}$ be the part process on the bounded domain $D$ of $Z= \{Z_t\}_{t \geq 0}$ with $Z_0=x \in D$. If there exists $\beta \in (0, 1)$ such that
\begin{align}\label{hc}
\int_{D} \mathbf{P}_x(T_D(Z^\dagger) \leq t) \, dx   \simeq   t^\frac{\beta}{2} \quad \textrm{as } t \to 0+ 
\end{align}
then we say that $Z$ exhibits a trap behaviour on $D$.
\label{defTraphc}
\end{definition}

Our reference model is given by the BM on the Dirichlet ball for which condition \eqref{hc} holds for $\beta=1$.

\begin{remark}\label{TAU}     If $Z$ equals in law the time-changed killed BM $X^\tau=X^\dagger \circ \tau$, then we say that $\tau$ introduces a trap behavior of $X^\dagger$ on $D$.
\end{remark}

The motivation of such characterization is based on the study of heat content
$Q(t)$  that serves as a global functional capable of decoding the geometry of a domain $\Omega$ through thermal dynamics.

We now recall some fundamental results.
  Let $D \subset \mathbb{R}^d$ be a bounded domain with boundary $\partial D$. We consider the initial-boundary value problem with Dirichlet conditions.
The standard heat equation is given by:
\begin{equation}
    \partial_t u(x,t) =\Delta u(x,t), \quad u(x,0)=1, \quad u(x,t)|_{\partial D}=0
\end{equation}
The heat content $Q(t)$ is the integral of the solution over the domain
\begin{equation}
    Q(t) = \int_{D} u(x, t) \, dx
\end{equation}
and the total amount of heat lost up to the moment $t$ is 
\begin{equation}
    Q_{loss}(t) = \int_{D} (1-u(x, t)) \, dx.
\end{equation}
The literature  regarding the heat content asymptotics for domains with Dirichlet boundary conditions  is huge.
If  $D\subset \R^n$ is a bounded connected domain with a $C^3$-regular boundary $\partial D$, then (as $t\to 0+$) the heat content of $D$ is determined up to the third order by 
$$
Q(t) = |D| - \sqrt{t}\frac{2}{\sqrt{\pi}} \mathcal{H}^{n-1}(\partial D)
+t\frac{n-1}{2}\int_{\partial D} H(x)\mathcal{H}^{n-1}(dx)+O(t^\frac{3}{2}),
$$
where $H(x)$ denotes the mean curvature at $x\in\partial D$, see \cite{Vdb94}.
In particular if $D\subset \R^2$  is a bounded connected domain with a $C^3$-regular boundary $\partial D$   \begin{equation}
    Q(t) = |D|  - \frac{2\sqrt{t}}{\sqrt{\pi}} |\partial D|  +  \int_{\partial D} H(s) \, ds \cdot t + O(t^{3/2}).
\end{equation}
We point out that the  third term is proportional to the integral of the mean curvature $H(s):$ 
this term explains how the \lq\lq bending" of the boundary accelerates or decelerates heat loss compared to a flat surface. When $\partial D$ contains corners (for example,  polygons), curvature $H$ is undefined at the vertices. In this case, 
when $D \subset \R^2$ is a bounded connected domain with a polygonal boundary $\partial D$, the following expansion, as $t\to 0+$, is given up to some exponential error by 
\begin{align*}
 Q(t) = |D|  -\sqrt{t}\frac{2}{\sqrt{\pi}}  |\partial D|  +t\sum_{i=1}^\ell c(\gamma_i)+O(e^{-\frac{r}{t}}),
\end{align*}
where $\ell$ is the number of vertices of $\partial\Omega$, $\gamma_i$, $i=1,\ldots, \ell$, are the interior angles at these vertices and 
\begin{align*}
    c(\gamma_i)=4\int_0^{+\infty} \frac{ \sinh ((\pi-\gamma_i)z)}{\sinh(\pi z)\cosh(\gamma_i z)}dz
\end{align*}
(see  \cite{Vdb90}). Each corner with interior angle $\gamma_i$ contributes a term proportional to $t$. The total angular contribution is $\sum c(\gamma_i)t$. Acute angles ($\gamma < \pi$) increase the rate of heat loss due to the proximity of multiple \lq\lq cold" boundaries.

Now we consider  the  Koch snowflake which has   a boundary with infinite length and non-integer dimension $d_f=\log 4/\log 3$.
 Because the perimeter is infinite, the standard $\sqrt{t}$ term is replaced by a power law dependent on the Minkowski dimension $d_f$ of the boundary. More precisely, when $D\subset \R^2$ is the  Koch snowflake $ \Omega_\mathsf{K}$ (with $a=\frac13,$) the  following  asymptotic expression holds as $t\to 0+$
$$ 
Q_{\Omega_\mathsf{K}}(t) = |\Omega_\mathsf{K}|-t^{\frac{2-d_f}{2}}p(\ln t) - t q(\ln t) + O(e^{-\frac{1}{1152t}}),$$
 as $t\to 0+$, where $d_f=\log 4/\log 3$ is the Minkowski dimension of $\partial  \Omega_\mathsf{K}$ and $p$ and $q$ are continuous $(\ln 9)$-periodic functions (see \cite{FLV}). The  previous expansion  is remarkable in the sense that it contains only a finite number of asymptotic terms with an exponential remainder estimate.
  In this respect the properties of the snowflake are somewhat similar to those of polygons, but while in sufficiently smooth boundaries the  heat content expansion  is in powers of $t^{1/2}$, fractal boundaries introduce non-integer powers and log-periodic oscillations. \\
  Then, by using the result in  \cite{FLV} we say that the heat diffusion on  $\Omega_\mathsf{K}$ exhibits a trap behaviour  according Definition \ref{hc} with $\beta={2-d_f}.$

  We  now consider in the plane    the following  Caputo fractional diffusion with  $\beta\in(0,1)$  of a disk  $\Omega_C$ of radius $R$
\begin{equation} \label{CD}
   D_t^{\beta} u(x,t) = \Delta u(x,t), \quad x \in \Omega_C \quad u(x,0)=1, \quad u(x,t)|_{\partial\Omega_C}=0.\end{equation}

  The corresponding time changed process  can be regarded as a prototype  as discussed in Remark \eqref{TAU} and it provides useful informations as the following theorem entails.

  \begin{theorem}\label{THC} Let $\Phi(\lambda) = \lambda^{\beta}$ characterize the time changed  process $X^L$  on $ \Omega_C.$
 According to Definition \ref{defTraphc},  the processes $X$ on ${\Omega}_\mathsf{K}$ and  the  time changed process $X^L$  on $ \Omega_C$   exhibit the same trap behaviour depending  on $\beta=2 - d_f$. 
\end{theorem}

\begin{proof}

For the  Koch snowflake  $\Omega_\mathsf{K}$  
we have that there exists a constant $C_K>0$  such that $$ C^{-1 }_K t^{\frac{2-d_f}{2}}  \leq Q_{loss, \Omega_\mathsf{K}}(t) \leq C_K t^{\frac{2-d_f}{2}}  $$ (see formula (0.7) in \cite{FLV}).

Now we consider the heat content for the  Caputo fractional diffusion \eqref{CD}.

By using the results for 
spectral heat content  in the framework of  time-changed
Brownian motions when the time changes are inverse subordinators,  we 
have  the following asymptotic expansion of
the spectral fractional heat content  as $t\to 0+$  (see Theorem 4.6 in \cite{KP}) 
   \begin{equation} Q^\beta_{  \Omega_C}(t) = \pi R^2     -\frac{2}{\sqrt{\pi}}    \frac{ {\Gamma(3/2)}}{\Gamma(1 + \beta/2)} 2\pi R\, t^{\frac{\beta}{2}}+ O(t^\beta)= \pi R^2     -  \frac{1}{\Gamma(1 + \beta/2)} 2\pi R\, t^{\frac{\beta}{2}}+ O(t^\beta).\end{equation}

Then we obtain that  as $t\to 0+$ for the solution $u$ to  the  Caputo fractional diffusion \eqref{CD}
 $$ Q^\beta_{loss, \Omega_C}(t)= \int_{\Omega_C} (1-u(x, t)) \, dx=    \frac{ 1}{\Gamma(1 + \beta/2)} 2\pi R\, t^{\frac{\beta}{2}}+ O(t^\beta).$$

Therefore, by choosing $\beta = 2 - d_f$ ,
$Q_{loss, \Omega_\mathsf{K}}(t) $ and $ Q^\beta_{loss, \Omega_C}(t)$ satisfy the same condition  \eqref{hc}.

\end{proof}

The previous Theorem \ref{THC}  says that, by mapping the fractal dimensions of the Koch snowflake onto a fractional exponent $\beta$, a  circle can be considered thermally indistinguishable from a fractal according the condition  \eqref{hc}.

\begin{figure}
    \centering
    \includegraphics[width=0.5\linewidth]{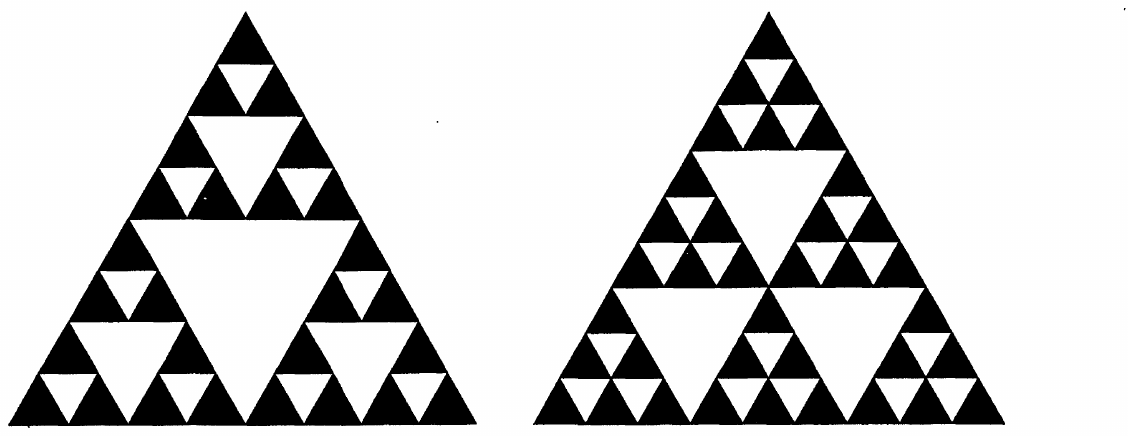}
    \caption{Sierpinski fractals $SG(2)$ and $SG(3)$}
    \label{fig:enter-label}
\end{figure}

\section{Sub-diffusive behaviours}

We recall that anomalous diffusion refers to a  process $Z$ that deviates from the linear growth of Mean Squared Displacement (MSD) over time that is  $$\mathbf{E}_x[(Z_t-x)^2] \sim  t^\beta$$ with $\beta\neq1.$
When $\beta=1$ we refer to a  a diffusion. In the case $0 < \beta < 1$ we have a sub-diffusive behaviour which may occur in systems with \lq\lq trapping" mechanisms or molecular crowding, often modeled by the Continuous Time Random Walk (CTRW) with heavy-tailed waiting time distributions.

We point out that also the process $X^+$ on Sierpinski gasket fractals $SG(2)$ has a sub-diffusive behaviour.
We recall the construction of  Sierpinski gasket $SG(2)$
Let $G_0 =\{a_0, a_1, a_2\} $  with $a_0=(0, 0),$ $ a_1=(1, 0)$ and $ a_2= (\frac12,\frac{\sqrt{3}}2)$   be the vertices of the unit
 triangle in $\R^2.$ 
The Sierpinski gasket  $SG(2)$ can be obtained from the general  theory of self-similar set as 
the unique compact set invariant with respect to the set of  the $3$ following contractive similarities with contraction factor $\frac12$
$$\psi_i(x)=a_i +\frac12(x-a_i), \quad\quad i=0,1, 2$$ 
that is $$SG(2)= \bigcup_{i_1,...,i_n=1}^4 \psi_i (SG(2)).$$

By Corollary 2.25 in \cite{BAR} we have on  $SG(2)$ that  $$\mathbf{E}_x[(X^+_t-x)^2] \sim  t^\beta,$$  with $\beta=2/d_{w_2}$ where the  walk dimension $d_{w_2}$  is   equal to  $\frac{ \log 5}{ \log 2}$
(we recall that  the  walk dimension relates mean exit times from balls with their radii,  that is, $$\mathbf{E}_x[T_B(X^+)] \sim \rho^{d_w}$$ (see \cite{BAR})).

As $2/d_{w_2} = \log 4/ \log 5 < 1$ we  conclude that the process $X^+$ on Sierpinski gasket fractals $SG(2)$ has a sub-diffusive behaviour.

In a similar manner, we consider $SG(3)$ a variant of  $SG(2).$ 
The set  $SG(3)$ that can be constructed 
by the same procedure seen before  by subdividing now  the
triangle  into $9$  smaller triangles, and deleting the  3 \lq\lq downward
facing" ones. More precisely, the construction of the variant Sierpinski gasket  $SG(3)$ can be obtained  by starting again from the triangle $G_0.$    Let $\{c_0, c_1\}$ be the points that divide the side with vertices $\{a_0, a_1\}$   of $G_0$ in third equal parts,  that is    $c_0=(\frac13, 0)$ $c_1=(\frac23, 0);$ 
let $\{c_2, c_3\}$ be the points that divide the side with vertices $\{a_1, a_2\}$   of $G_0$  in third equal parts, that is    $c_3=(\frac23,\sqrt{3}/3 )$ $c_2=(5/6,  \sqrt{3}/6),$ 
 let $\{c_4, c_5\}$ be the points that divide the side with vertices $\{a_0, a_2\}$   of $G_0$ in third equal parts, that is    $c_5=(\frac16,\sqrt{3}/6 )$ $c_4=(1/3,  \sqrt{3}/3).$ Moreover we consider the point $c_6=(1/2,  \sqrt{3}/6).$
 Let $A_1$ be the interior of the triangle with vertices $\{c_0, c_6, c_5\};$  let $A_2$ be the interior of the triangle with vertices $\{c_1, c_2, c_6\}$ and
  let $A_3$ be the interior of the triangle with vertices $\{c_6, c_3, c_4\}.$  Let  $H_1 = H_0 \setminus  (A_1\cup A_2\cup A_3),$ so $H_1$ consists of $6$ closed upward facing triangles, each of side $\frac13.$
Now repeat the operation on each of these triangles to obtain the set $H_2$ consisting of $36$ upward
 closed upward facing triangles, each of side $\frac19.$
 Continuing in this way, we obtain a decreasing sequence of closed non-empty
sets $H_n$, and set
$SG(3)=\cap^{\infty}_{n=0} H_n.$
Moreover, also the Sierpinski gasket $SG(3)$ can be defined as the unique compact set invariant with respect to the set of the following $N=6$ contractive similarities with contraction factor $\frac13:$
$$\varphi_i(x)=a_i +\frac13(x-a_i), \quad\quad i=0,1,2, $$
$$ \varphi_3(x)=\frac13 x +c_0,  \varphi_4(x)=\frac13 x +c_6 , \varphi_5(x)=\frac13 x +c_5.$$ 

Also on $SG(3)$ we have  a sub-diffusive behaviour, that is,    $$\mathbf{E}_x[(X^+_t-x)^2] \sim  t^\beta$$  $\beta=2/d_{w_3}$ with walk dimension   
$d_{w_3}= \frac{\log(\frac{90}7)}{\log 3}.$

We observe that on these fractals,    (see, for example \cite[Example 6.4]{CapDovDelRus})  the time changed process $X^+ \circ L$ (with  $\Phi(\lambda) = \lambda^{\beta}$) exhibits a sub-diffusive behaviour for any value of $\beta\in(0,1)$ since (see, for example \cite[Example 6.4]{CapDovDelRus})
\begin{align*}
\mathbf{E}_0[(X^+ \circ L_t)^2] = \mathbf{E}[(L_t)^{2/d_{w}}] \sim t^ {2\beta / d_{w}}
\end{align*}
(where $d_w$ is the walk dimension).

Here we recall the result obtained in \cite{RCAMS} where the connection between complex structures and non-local operators has been obtained for variants of the Sierpinski gasket.   
In particular, it is proved that the  same on-diagonal heat kernel estimate \begin{equation}\label{hkd} q(t, x,x) \simeq t^{-(\frac{\log 6}{\log 3})/\frac{\log(\frac{90}7)}{\log 3}} \end{equation}\
can be related  either to a time fractional diffusion on Sierpinski gasket  $SG(2)$ with $$\beta =\frac{1-\frac{\log 6}{\log(\frac{90}7)}}{1-\frac{\log3}{\log5}}$$ 
 or  a  diffusion on the  variant of Sierpinski gasket $SG(3)$.

 In particular, this  underlines the link between the features of the medium and the anomalous diffusions which are well-described by time changed processes
and it shows as the sub-diffusive behavior of the process on  $SG(3)$  can be related to the sub-diffusive behavior  ot the change timed process on the \textit{simpler}  Sierpinki gasket $SG(2).$
We remark that the integral on the domain of  on-diagonal heat kernel gives the so-called heat trace or partition function (see \cite{Vdb87}) where the walk dimension $d_w$ is needed to characterize the log-periodic oscillations appearing in the fractal case (see \cite{DUNNE}): a more detailed focus could be performed by using variable order fractional operator.

\section{Conclusion}

Our tools provide a deep investigation  on trap behaviour which surprising conclusions. For example, for the Koch snowflakes $\Omega_\mathsf{K}$ we have that:
\begin{itemize}
\item[i)] $\Omega_\mathsf{K}$ is non trap for $X^+$ according to Definition \eqref{defTrapd},
\item[ii)] $X^+$ does not exhibit trap behaviour on $\Omega_\mathsf{K}$ according to Definition \eqref{defTrapb},
\item[iii)] $X^+$ exhibits trap behaviour on $\Omega_\mathsf{K}$ according to Definition \eqref{defTraphc}.
\end{itemize}
Thus, the arguments under investigation provide qualitatively different information. For example, we have that:
\begin{itemize}
\item[i)] $X^+$ on $\mathsf{SG}$ exhibits a sub-diffusive behaviour and does not exhibit trap behaviour according to Definition \ref{defTrapb}. 
\item[ii)] $\bar{X}$ on $\overline{D}$ exhibits trap behaviour according to Definition \ref{defTrapb} and does not exhibit sub-diffusive behaviour on $D$.
\item[iii)] $\bar{X}$ on $\mathsf{G}$ exhibits trap behaviour according to Definition \ref{defTrapb} and does not exhibit sub-diffusive behaviour on $\mathsf{G} \setminus \mathcal{V}$.
\end{itemize}

Finally, we underline that:
\begin{itemize}
\item[i)] $X^L$ has always sub-diffusive and trap behaviours. 
\end{itemize}

These tools can be have  significant practical applications in characterizing irregular media.
In particular, it could be useful for non-destructive thermal diagnostics. In fact, 
in the field of material science, the structural integrity of porous thermal barriers can be assessed by observing surface cooling curves. By fitting experimental data to a fractional model on a simple geometry, the value of $\beta$ allows for the identification of internal fractal dimensions without invasive imaging. 
An other application could be electrochemical energy storage.
The diffusion of ions in fractal electrodes follows a similar mathematical structure to heat flow. The  fractional model can simulate the impedance of fractal electrodes using smooth-geometry models, significantly reducing computational cost while maintaining accuracy regarding the electrode's surface roughness \cite{pajkossy}.
Moreover, fractal structures are ubiquitous in biological systems, such as the pulmonary tree and vascular networks. Modeling these structures  as more regular domains with fractional suitable diffusions provides a  framework for interpreting and approximating  the process on these   (see \cite{magin}).

 \section*{Acknowledgments}

The authors are supported by INdAM-GNAMPA and Sapienza University of Rome. 

The authors thank MUR for the support under the project PRIN 2022 - 2022XZSAFN: Anomalous Phenomena on Regular and Irregular Domains: Approximating Complexity for the Applied Sciences - CUP B53D23009540006 - PNRR M4.C2.1.1

\small{

}

\end{document}